\documentclass[12pt]{amsart}

\usepackage{amsmath,amsfonts,amssymb,amsthm}
\usepackage{pinlabel}
\usepackage{graphicx}
\usepackage{mathtools}
\usepackage[all]{xy}
\usepackage{hyperref}
\usepackage[usenames,dvipsnames]{color}
\usepackage{color}
\xyoption{dvips}

\numberwithin{equation}{section}

\newtheorem{thm}{Theorem}[section]

\newtheorem{prop}[thm]{Proposition}
\newtheorem{lem}[thm]{Lemma}

\theoremstyle{remark}

\theoremstyle{definition}

\theoremstyle{remark}
\newtheorem{remark}[thm]{Remark}
\theoremstyle{property}

\newcommand{\F}{\mathbb{F}}

\newcommand{\MS}{\mathbb{S}}

\newcommand{\im}{\operatorname{im}}

\begin{document}

\title{Non-decomposable Lagrangian endoconcordances and Khovanov homology}

\author{Roman Golovko}

\begin{abstract}

In this note, we construct non-decomposable Lagrangian endoconcordances of
Legendrian knots. We prove that for any positive integer $k$, there exist a (stabilised) Legendrian knot  and  a Lagrangian fillable Legendrian knot 
in the standard contact $3$-sphere, each admitting at least $k$ pairwise Hamiltonian non-isotopic
relative boundary, non-decomposable Lagrangian endoconcordances. The obstruction
to decomposability that we rely on is based on Khovanov homology.
\end{abstract}

\address{Faculty of Mathematics and Physics, Charles University, Sokolovsk\'{a} 83, 18000 Praha 8, Czech Republic}
\email{golovko@karlin.mff.cuni.cz}
\date{\today}



\maketitle

\section{Introduction and main results}

Exact Lagrangian cobordisms between Legendrian submanifolds in Weinstein cobordisms are central objects in modern symplectic topology. They naturally fit into the symplectic field theory framework of Eliashberg--Givental--Hofer \cite{IntroductionSFT}. They have been actively studied using various methods for more than two decades \cite{GeogrBourSabloffTraynor, CasalsGaoInfinite, LagrConcLegKnots, LagrConcNotSymRel, FlThLagCob, ObstrLagrConc2016, LegKnotsLagrCob12, EliashbergMurphyLagCaps, EtnyreLeversonribbon, Notpartordhighdim,LinExLagrCaps,KhovanovMPWLagrangian}.

Decomposable Lagrangian cobordisms form a particularly nice and useful class of  exact Lagrangian cobordisms in the symplectization of the standard contact $3$-sphere $S^3_{st}$. 
They were defined in the works of Chantraine \cite{Non-collarableSlices} and Ekholm--Honda--Kálmán \cite{LegKnotsLagrCob12}.
Recall that a Lagrangian cobordism in the symplectization of $S^3_{st}$
 is said to be decomposable if it can be presented as a concatenation of the following simple types of Lagrangian cobordisms:
\begin{itemize}
\item trace cobordisms induced by Legendrian isotopies;
\item a standard Lagrangian disk filling of the $tb=-1$ unknot;
\item pair-of-pants Lagrangian cobordisms arising from pinch moves.
\end{itemize}
Observe that every decomposable Lagrangian cobordism is automatically ribbon. In other words, it can be viewed as a Morse cobordism whose Morse function has only index $0$ and $1$ critical points.
The first examples of non-decomposable Lagrangian cobordisms between Legendrian knots were Lagrangian caps, which were studied by Lin \cite{LinExLagrCaps}. For a while, it was an open question whether every exact Lagrangian cobordism with non-empty positive end is decomposable.

The first constructions of non-decomposable Lagrangian cobordisms with non-empty positive ends are very recent \cite{NonregStabilizedConc2026, NonregFillable2025,Nondecompcobordisms2025}, and they all rely on the approximation result of Dimitroglou Rizell \cite{LagrApproxTotReal}. In all these constructions, it is crucial that the positive and negative Legendrian ends are smoothly non-isotopic. In this paper, we construct examples of non-decomposable Lagrangian endoconcordances, i.e., exact Lagrangian cobordisms diffeomorphic to cylinders whose Legendrian ends coincide not only as smooth knots but also as Legendrian knots.

We also note that, in our previous works on non-decomposable Lagrangian cobordisms with Dimitroglou Rizell \cite{NonregStabilizedConc2026, NonregFillable2025} and Komarek \cite{Nondecompcobordisms2025}, we relied on obstructions to decomposability arising from the work of Cornwell--Ng--Sivek \cite{ObstrLagrConc2016}, Agol's result showing that ribbon (or strongly homotopy-ribbon) concordances define a partial order \cite{RibbonConcPartialOrdering}, and Livingston's estimates \cite{LivingstonCritPoints2023}. In this note, the obstruction to decomposability that we use is different: it comes from the behavior of the Khovanov homology functor under ribbon concordances, as investigated by Levine--Zemke in \cite{KhovanovHomRibbon}.

In addition to constructing non-decomposable Lagrangian endoconcordances, we prove that there are stabilised Legendrian knots and Lagrangian fillable Legendrian knots whose sets of Lagrangian endoconcordances contain as many pairwise Hamiltonian non-isotopic relative boundary, non-decomposable Lagrangian endoconcordances as possible. More precisely, our main result says the following:
\begin{thm}
\label{differentendo}
For Legendrian knots in the standard contact $3$-sphere $S^3_{st}$, the following two statements hold:
\begin{itemize}
\item[(1)] For any positive integer $k$,  there is a (stabilised) Legendrian knot $\Lambda_k$ with the property that the set of Lagrangian endoconcordances of $\Lambda_k$
contains at least $k$ pairwise Hamiltonian non-isotopic relative boundary, non-decomposable Lagrangian endoconcordances.
\item[(2)] For any positive integer $k$,  there is a Lagrangian fillable Legendrian knot $\Lambda_{fil, k}$ with the property that the set of Lagrangian endoconcordances of $\Lambda_{fil, k}$
contains at least $k$ pairwise Hamiltonian non-isotopic relative boundary, non-decomposable Lagrangian endoconcordances. 
\end{itemize}
\end{thm}

\begin{remark}
Note that recently non-isotopic endocobordisms have been used by  Dimitroglou Rizell and Lawrence in their construction of knotted surfaces in $\mathbb C^2$ \cite{DimitroglouRizellLawrence2025}. We are not relying on that work. Dimitroglou Rizell and Lawrence do not discuss decomposability phenomenon for Lagrangian cobordisms in \cite{DimitroglouRizellLawrence2025}.  The technique used in \cite{DimitroglouRizellLawrence2025} is completely different from the technique we use in this paper.
\end{remark}

\section*{Acknowledgements}
We would like to thank Georgios Dimitroglou Rizell and Maciej Borodzik for the very helpful discussions. The author is supported by the GAČR grant 26-20231L.

\section{Non-decomposable endoconcordance}
In this section we explain how to get a non-decomposable Lagrangian endoconcordance from a decomposable Lagrangian concordance whose ends have different  Khovanov homologies over $\F_2$. The following proposition will be crucial for the proof of Thorem \ref{differentendo}.
\begin{prop}
\label{non-decompoosableendo}
Let $\Lambda_-$ and $\Lambda_+$ be two Legendrian knots with different Khovanov homologies over $\mathbb F_2$. 
In other words, suppose that there exists a pair $(i_0,j_0)$ such that $$Kh^{i_0,j_0}(\Lambda_-, \mathbb F_2)\ncong Kh^{i_0,j_0}(\Lambda_+, \mathbb F_2).$$Suppose, in addition, that there exists a decomposable Lagrangian concordance $L$ from   $\Lambda_-$ to 
$\Lambda_+$. Then there exists a non-decomposable Lagrangian endoconcordance of  $S(\Lambda_+)$, where $S(\Lambda_+)$ denotes a sufficiently stabilized Legendrian knot obtained from $\Lambda_+$. 
\end{prop}

\begin{proof}
First observe that from \cite[Theorem 3.1]{NonregFillable2025} it follows that there are two Lagrangian concordances $S(L)$ from $S(\Lambda_-)$ to
$S(\Lambda_+)$ and $S(T)$ from $S(\Lambda_+)$  to $S(\Lambda_-)$. Here $S(\Lambda_{\pm})$ are obtained from $\Lambda_\pm$ by sufficiently
many stabilisations. Since stabilistions do not change the corresponding isotopy classes, we get $Kh^{i,j}(S(\Lambda_{\pm}), \mathbb F_2)=Kh^{i,j}(\Lambda_{\pm}, \mathbb F_2)$ for all $(i,j)$, and hence, in particular, $Kh^{i_0,j_0}(S(\Lambda_-), \mathbb F_2)\ncong Kh^{i_0,j_0}(S(\Lambda_+), \mathbb F_2)$. Then we consider the concatenation of Lagrangian concordances $S(L)\circ S(T)$ (meaning $S(T)$
followed by $S(L)$)which is an endoconcordance of $S(\Lambda_+)$. Assume that $S(L)\circ S(T)$ is a ribbon endoconcordance.
Then from the result of Levine--Zemke \cite[Theorem 1]{KhovanovHomRibbon} it follows that the map $$Kh^{i,j}(S(L)\circ S(T)): Kh^{i,j}(S(\Lambda_+), \mathbb F_2)\to Kh^{i,j}(S(\Lambda_+), \mathbb F_2)$$ 
is an injective homomorphism for all $(i,j)$. Now, we note that by the functoriality of Khovanov homology $Kh(S(L)\circ S(T))=Kh(S(L))\circ Kh(S(T))$. 
Since $Kh^{i,j}(S(L)\circ S(T))$ is assumed to be injective for all $(i,j)$, $Kh^{i,j}(S(T))$ must be injective as well. 
This implies that 
\begin{align}
\label{sideTinj}
\dim Kh^{i,j}(S(\Lambda_+), \mathbb F_2) \leq \dim Kh^{i,j}(S(\Lambda_-), \mathbb F_2)
\end{align}
for all $(i,j)$.
We now recall that $L$ is decomposable, and hence ribbon. Then $S(L)$ is also ribbon, and hence $Kh^{i,j}(S(L))$ is injective. This implies that 
\begin{align}
\label{sideLinj}
\dim Kh^{i,j}(S(\Lambda_-), \mathbb F_2) \leq \dim Kh^{i,j}(S(\Lambda_+), \mathbb F_2)
\end{align}
for all $(i,j)$.
We now observe that Inequalities \ref{sideTinj} and \ref{sideLinj} imply that $Kh^{i,j}(S(\Lambda_-), \mathbb F_2)\cong Kh^{i,j}(S(\Lambda_+), \mathbb F_2)$ for all $(i,j)$, which contradicts the fact that there exists $(i_0,j_0)$ such that
$Kh^{i_0,j_0}(\Lambda_-, \mathbb F_2)\ncong Kh^{i_0,j_0}(\Lambda_+, \mathbb F_2)$. Hence we see that $S(L)\circ S(T)$ is a non-ribbon endoconcordance, and therefore $S(L)\circ S(T)$ is a non-decomposable Lagrangian endoconcordance of $S(\Lambda_+)$.
\end{proof}

\begin{remark}
In Proposition \ref{non-decompoosableendo}, we assume that the ends of the decomposable Lagrangian concordance have different Khovanov homologies over $\mathbb F_2$. If, instead of Khovanov homology, we use knot Floer homology together with \cite[Theorem 1.1]{KFHribbon}, then the analogous statement involving knot Floer homology also holds.
\end{remark}

\section{Proof of Theorem \ref{differentendo}}

Let $\Lambda$ be a Legendrian knot in $S^3_{st}$ which admits a decomposable Lagrangian concordance $L_{U,\Lambda}$ from $tb=-1$ unknot $U$ to $\Lambda$. Let $\Lambda^{k}:=\Lambda\#\dots\#\Lambda$ denote the $n$-times connected sum of $\Lambda$.
Then we consider the following sequence of decomposable Lagrangian concordances: 
$$U\prec^{dec}_{L_{U,\Lambda}}\Lambda\prec^{dec}_{L_{\Lambda,\Lambda^2}}\Lambda^2\prec^{dec}_{L_{\Lambda^2,\Lambda^3}}\dots\prec^{dec}_{L_{\Lambda^{k-1},\Lambda^k}}\Lambda^k.$$

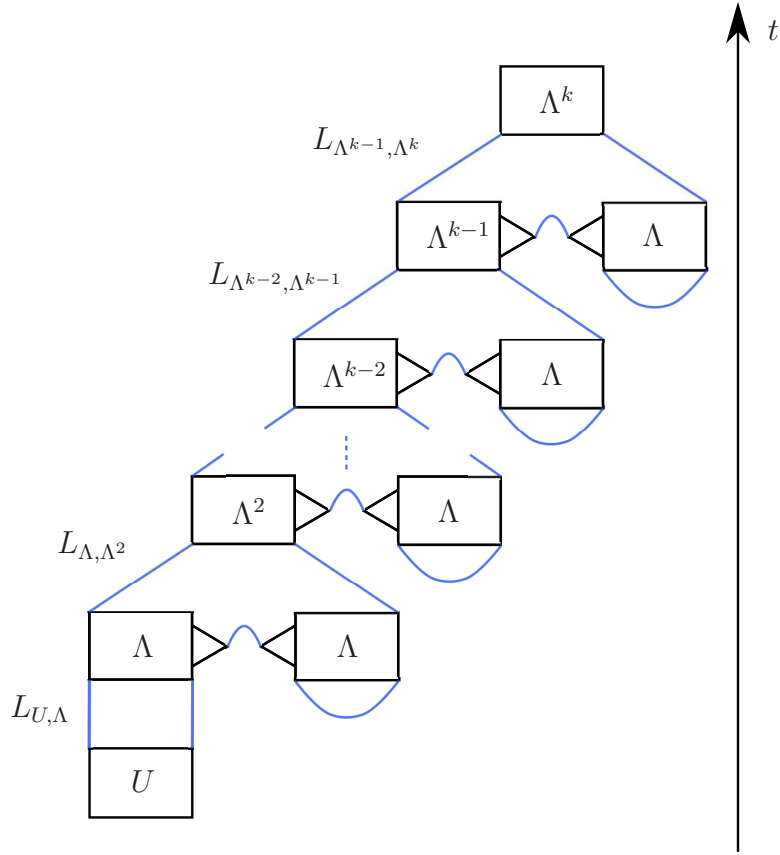
\begin{figure}[t]
\begin{center}
\vspace{3mm}
\labellist
\pinlabel $t$ at 425 775
\pinlabel $\Lambda^k$ at 335 746
\pinlabel $L_{\Lambda^{k-1},\Lambda^{k}}$ at 256 728
\pinlabel $\Lambda$ at 375 690
\pinlabel $\Lambda^{k-1}$ at 294 690
\pinlabel $L_{\Lambda^{k-2},\Lambda^{k-1}}$ at 218 672
\pinlabel $\Lambda$ at 333 632
\pinlabel $\Lambda^{k-2}$ at 251 632
\pinlabel $\Lambda$ at 290 575
\pinlabel $\Lambda^2$ at 207 575
\pinlabel $L_{\Lambda,\Lambda^2}$ at 142 560
\pinlabel $\Lambda$ at 248 519
\pinlabel $\Lambda$ at 163 519
\pinlabel $U$ at 163 463
\pinlabel $L_{U,\Lambda}$ at 120 492
\endlabellist
\includegraphics[height=320px]{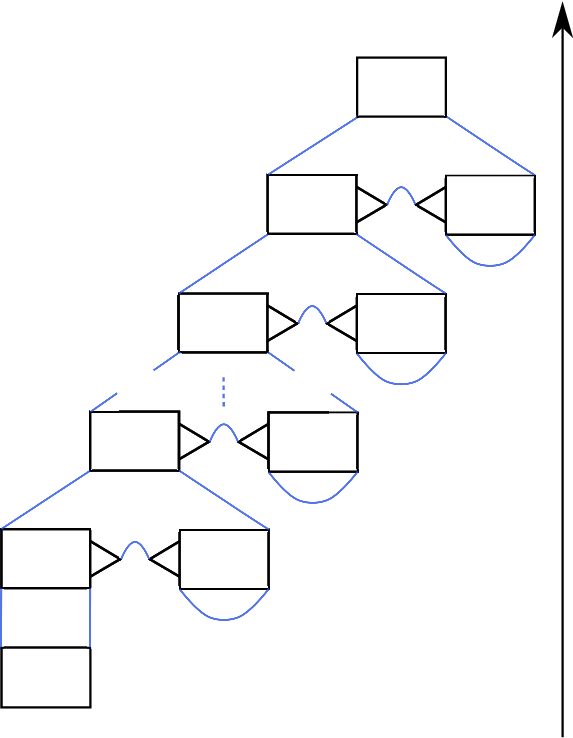}
\caption{The schematic picture of the concatenated decomposable Lagrangian concordances $L_{U,\Lambda}$ and $L_{\Lambda^{i-1},\Lambda^{i}}$, $2\leq i \leq k$.} 
\label{fig:longconcordance}
\end{center}
\end{figure}

The schematic picture of the decomposable Lagrangian concordances $L_{U,\Lambda}$ and $L_{\Lambda^{i-1},\Lambda^{i}}$, $2\leq i \leq k$, appears in Figure \ref{fig:longconcordance}. 
Then from \cite[Theorem 3.1]{NonregFillable2025} it follows that  there are Lagrangian concordances $S_1(L_{U,\Lambda})$ from $S_1(U)$ to
$S_1(\Lambda)$ and $S_1(T_{U,\Lambda})$ from $S_1(\Lambda)$ to $S_1(U)$, and Lagrangian concordances 
$S_i(L_{\Lambda^{i-1}, \Lambda^i})$ from $S_i(\Lambda^{i-1})$ to $S_i(\Lambda^i)$  and 
$S_i(T_{\Lambda^{i-1}, \Lambda^i})$ from  $S_i(\Lambda^i)$ to $S_i(\Lambda^{i-1})$, $2\leq i \leq k$. Here $S_i$'s denote multiple pairs of stabilisations that are required by \cite[Theorem 3.1]{NonregFillable2025}. 

Assume that $S_i$ consist of $m_i$ pairs of stabilisations. We now define $\mathbb S$ to be the sequence of $\max_{i\in\{1,\dots, k\}}(m_i)$ pairs of stabilisations.  Then we observe that from the fact that there are Lagrangian concordances $S_1(L_{U,\Lambda})$ and $S_1(T_{U,\Lambda})$ with 
$S_1(L_{U,\Lambda})$ being ribbon and \cite[Proposition 5.6]{LagrConcLegKnots} it follows that there are Lagrangian concordances $\MS(L_{U,\Lambda})$ and $\MS(T_{U,\Lambda})$ with $\MS(L_{U,\Lambda})$  being ribbon. Similarly, from the fact that there are Lagrangian concordances $S_i(L_{\Lambda^{i-1},\Lambda^i})$ and $S_i(T_{\Lambda^{i-1},\Lambda^i})$ with 
$S_i(L_{\Lambda^{i-1},\Lambda^i})$ being ribbon and \cite[Proposition 5.6]{LagrConcLegKnots} it follows that there are Lagrangian concordances $\MS(L_{\Lambda^{i-1},\Lambda^i})$ and $\MS(T_{L_{\Lambda^{i-1},\Lambda^i}})$ with $\MS(L_{\Lambda^{i-1},\Lambda^i})$ being ribbon. 

We then get the following two sequences of Lagrangian cobordisms:
\begin{align}\nonumber
&\MS(U)\prec_{\MS(L_{U,\Lambda})}\MS(\Lambda)\prec_{\MS(L_{\Lambda,\Lambda^2})}\MS(\Lambda^2)\prec_{\MS(L_{\Lambda^2,\Lambda^3})}\dots\prec_{\MS(L_{\Lambda^{k-1},\Lambda^k})}\MS(\Lambda^k)\ \mbox{and}\\
&\nonumber\MS(U)\succ_{\MS(T_{U,\Lambda})}\MS(\Lambda)\succ_{\MS(T_{\Lambda,\Lambda^2})}\MS(\Lambda^2)\succ_{\MS(T_{\Lambda^2,\Lambda^3})}\dots\succ_{\MS(T_{\Lambda^{k-1},\Lambda^k})}\MS(\Lambda^k).
\end{align}
We now construct the following sequence of Lagrangian concordances that we call the $i$th concordance sequence of $\MS(\Lambda^k)$, $i=0,\dots,k-1$ and denote it by $CS_i(\MS(\Lambda^k))$:
$$\MS(\Lambda^k)\prec_{\MS(T_{\Lambda^k,\Lambda^{k-1}})}\MS(\Lambda^{k-1})\dots\prec_{\MS(T_{\Lambda^{i+1},\Lambda^i})}\MS(\Lambda^i)\prec_{\MS(L_{\Lambda^{i},\Lambda^{i+1}})}\MS(\Lambda^{i+1})\dots \prec_{\MS(L_{\Lambda^{k-1},\Lambda^k})}\MS(\Lambda^k).$$
Here for simplicity we write $\MS(\Lambda^0)$ for $\MS(U)$.

In addition, we observe that from \cite[Theorem 3.1]{NonregFillable2025} and \cite[Proposition 5.6]{LagrConcLegKnots} it follows that $\MS(T_{U,\Lambda})$ is smoothly isotopic relative boundary to   $\overline{\MS(L_{U,\Lambda})}$, where  $\overline{\MS(L_{U,\Lambda})}$ denotes the reverse concordance of $L_{U,\Lambda}$. Using the same argument, we see that $\MS(T_{\Lambda^{i-1},\Lambda^i})$ is smoothly isotopic relative boundary to   $\overline{\MS(L_{\Lambda^{i-1}, \Lambda^i})}$. Hence we can say that the left part of  $CS_i(\MS(\Lambda^k))$ that starts at $\MS(\Lambda^k)$ and stops at $\MS(\Lambda^i)$ is smoothly isotopic relative boundary to the reverse of the right part of  $CS_i(\MS(\Lambda^k))$ that starts at  $\MS(\Lambda^i)$ and ends at  $\MS(\Lambda^k)$. We denote the lefts part by $LP_i(\MS(\Lambda^k))$ and the right part by $RP_i(\MS(\Lambda^k))$. The previous discussion implies that $LP_i(\MS(\Lambda^k))$ is smoothly isotopic relative boundary to  $\overline{RP_i(\MS(\Lambda^k))}$. Since  $RP_i(\MS(\Lambda^k))$ is a ribbon concordance as a concatenation of ribbon concordances, following \cite[Theorem 1]{KhovanovHomRibbon} we see that the induced map in Khovanov homology  
$$Kh(RP_i(\MS(\Lambda^k))): Kh(\MS(\Lambda^i))\to Kh(\MS(\Lambda^k))$$ is injective. 
Now since $LP_i(\MS(\Lambda^k))$ is smoothly isotopic relative boundary to  $\overline{RP_i(\MS(\Lambda^k))}$, we following 
 \cite[Theorem 1]{KhovanovHomRibbon} observe that the induced map in Khovanov homology  
$$Kh(LP_i(\MS(\Lambda^k))): Kh(\MS(\Lambda^k))\to Kh(\MS(\Lambda^i))$$ is surjective. 
Then we apply the Khovanov homology functor $Kh$ to $CS_i(\MS(\Lambda^k))$ and get the following diagram:
$$Kh(CS_i(\MS(\Lambda^k))):Kh(\MS(\Lambda^k))\twoheadrightarrow Kh(\MS(\Lambda^i))\hookrightarrow Kh(\MS(\Lambda^k)).$$
Since the first map is surjective, the second map is injective, and we consider  Khovanov homology functor over $\mathbb F_2$, observe that 
\begin{align}\label{dimdistinction}
\dim \im (Kh(CS_i(\MS(\Lambda^k))))=\dim Kh(\MS(\Lambda^i))=\dim Kh(\Lambda^i).
\end{align}
Here dimension stands for the total dimension.
We now observe that since $\Lambda$ is knotted, by the result of Kronheimer--Mrowka \cite{KronheimerMrowkaunknot} it follows that $\dim Kh(\Lambda; \mathbb F_2)\geq 4$. 
Then we need the following basic computation of the total dimension of Khovanov homology of $\Lambda^k$ over $\mathbb F_2$. 
\begin{lem}
\label{kunnethkhovanovunr}
Let $\Lambda$ be a knotted knot in $S^3$ and $\dim Kh(\Lambda; \mathbb F_2)=d$. Then the total dimension of the Khovanov homology of $\Lambda^k$ can be computed as $\dim Kh(\Lambda^k; \mathbb F_2)=\frac{d^k}{2^{k-1}}$ for $k\geq 1$, and hence $$\dim Kh(\Lambda^k; \mathbb F_2)> \dim Kh(\Lambda^i; \mathbb F_2)>\dim Kh(U; \mathbb F_2)$$
for $1\leq i\leq k-1$, $k\geq 2$.
\end{lem}
\begin{proof}
We first observe that from \cite[Theorem 3.2.A, Corollary 3.2.C]{Shumakovichtorsion} it follows that $$d=\dim Kh(\Lambda; \mathbb F_2)=2\dim \widetilde{Kh}(\Lambda; \mathbb F_2),$$ and therefore $\dim \widetilde{Kh}(\Lambda; \mathbb F_2)=d/2$.
Here $\widetilde{Kh}(\Lambda^k; \mathbb F_2)$ denotes the reduced Khovanov homology of $\Lambda^k$.
Then by applying the K\"{u}nneth formula for reduced Khovanov homology \cite[Lemma 2.15]{DunfieldLipshitzSchutzKunneth} we observe that $$\dim \widetilde{Kh}(\Lambda^k; \mathbb F_2)=(\dim \widetilde{Kh}(\Lambda; \mathbb F_2))^k.$$ 
Thus we get $$\dim Kh(\Lambda^k; \mathbb F_2)=2\dim \widetilde{Kh}(\Lambda^k; \mathbb F_2)=2(\dim \widetilde{Kh}(\Lambda; \mathbb F_2))^k=2\frac{d^k}{2^k}=
\frac{d^{k}}{2^{k-1}}.$$
Now recall that from the result of Kronheimer and Mrowka \cite{KronheimerMrowkaunknot} it follows that $$\dim Kh(U; \mathbb F_2)=2\ \mbox{and}\ \dim Kh(\Lambda; \mathbb F_2)=d\geq 4,$$ that together with the computation of $\dim Kh(\Lambda^k; \mathbb F_2)$ leads to the desired inequality.
\end{proof}

We now observe that Formula \ref{dimdistinction} together with Lemma \ref{kunnethkhovanovunr}  imply that  $$\dim \im(Kh(CS_i(\MS(\Lambda^k))))=\dim Kh(\Lambda^i; \mathbb F_2)\neq \dim Kh(\Lambda^j; \mathbb F_2)=\dim \im(Kh(CS_j(\MS(\Lambda^k))))$$ for $i\neq j$.
Thus  Lagrangian concordance defined by  $CS_i(\MS(\Lambda^k)))$ is not smoothly isotopic relative boundary to the one defined by $CS_j(\MS(\Lambda^k))$ for $i\neq j$. This leads to the fact that Lagrangian concordance defined by  $CS_i(\MS(\Lambda^k)))$ is not Hamiltonian isotopic relative boundary to the one defined by $CS_j(\MS(\Lambda^k))$ for $i\neq j$. Finally, we see that  
$\Lambda_k:=\MS(\Lambda^k)$ admits $k$ pairwise Hamiltonian non-isotopic relative boundary Lagrangian concordances defined by $CS_i(\MS(\Lambda^k)))$, $i=0,\dots k-1$, that are non-decomposable by Proposition \ref{non-decompoosableendo}.  
This finishes the proof of the first part of Theorem \ref{differentendo}.

\begin{figure}[t]
\includegraphics[height=5cm]{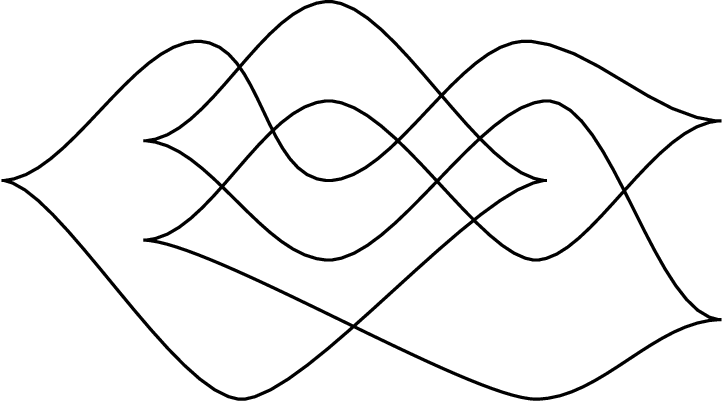}
\vspace{0.1in}
\caption{Front projection of $\Lambda_{m(9_{46})}$.}
\label{fig:m946m}
\end{figure}

Then we prove the second part of Theorem \ref{differentendo}. For that we would like to partially translate our construction from the proof of the first  part that concerns stabilised Legendrian knots to the class of Lagrangian fillable Legendrian knots. We consider $\Lambda=\Lambda_{m(9_{46})}$, the front projection of it appears in Figure \ref{fig:m946m}. Recall that there is a decomposable Lagrangian concordance from $U$ to $\Lambda$, it has been discussed, in particular, in \cite{NonregStabilizedConc2026}. 
Now, as in the proof of the first part, we construct the following sequence $CS_i(\MS(\Lambda^k))$, $i=0,\dots,k-1$, and apply the Legendrian Whitehead double to it

\begin{align*}
&Wh(\MS(\Lambda^k))\prec_{Wh(\MS(T_{\Lambda^k,\Lambda^{k-1}}))}\dots\prec_{Wh(\MS(T_{\Lambda^{i+1},\Lambda^i}))}Wh(\MS(\Lambda^i))\prec_{Wh(\MS(L_{\Lambda^{i},\Lambda^{i+1}}))}\\&\dots \prec_{Wh(\MS(L_{\Lambda^{k-1},\Lambda^k}))}Wh(\MS(\Lambda^k)).\end{align*}
We call the resulting sequence by $Wh(CS_i((\MS(\Lambda^k)))$, $i=0, \dots, k-1$. Recall that  Legendrian Whitehead doubling corresponds to taking Legendrian satteliting with the pattern $Wh$, whose front diagram appears in Figure \ref{fig:Whitehead_both}. 
From the work of Bourgeois–Sabloff–Traynor \cite[Section 5.1, Proposition 5.1]{GeogrBourSabloffTraynor} it follows that $Wh(\MS(\Lambda^i))$, $i=0,\dots,k$,  admits genus one Lagrangian filling. 
Observe that since $\MS(L_{U,\Lambda})$ and $\MS(L_{\Lambda^{i-1},\Lambda^i})$ are ribbon, $Wh(\MS(L_{U,\Lambda}))$ and $Wh(\MS(L_{\Lambda^{i-1},\Lambda^i}))$ are ribbon. Besides that, since $\MS(T_{U,\Lambda})$ is smoothly isotopic relative boundary to   $\overline{\MS(L_{U, \Lambda})}$ and $\MS(T_{\Lambda^i,\Lambda^{i+1}})$ is smoothly isotopic relative boundary to   $\overline{\MS(L_{\Lambda^i,\Lambda^{i+1}})}$, $Wh(\MS(T_{U,\Lambda}))$ is smoothly isotopic relative boundary to   $\overline{Wh(\MS(L_{U, \Lambda}))}$ and $Wh(\MS(T_{\Lambda^i,\Lambda^{i+1}}))$ is smoothly isotopic relative boundary to   $\overline{Wh(\MS(L_{\Lambda^i,\Lambda^{i+1}}))}$. 
Therefore the previous discussion implies that $LP_i(Wh(\MS(\Lambda^k)))$ is smoothly isotopic relative boundary to  $\overline{Wh(RP_i(\MS(\Lambda^k)))}$.

So, in order to apply to  $Wh(CS_i((\MS(\Lambda^k)))$ the same argument we applied above and to say using Proposition \ref{non-decompoosableendo} that $Wh(CS_i((\MS(\Lambda^k)))$ define 
pairwise Hamiltonian non-isotopic relative boundary, non-decomposable Lagrangian endoconcordances, we need to show that $Wh(\MS(\Lambda^i))$, $i=0,\dots,k-1$ have pairwise different Khovanov homologies over $\F_2$. 

\begin{figure}[t]
\includegraphics[height=1.9cm]{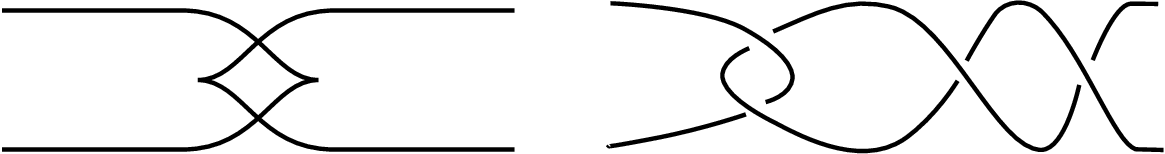}
\vspace{0.1in}
\caption{Left: front projection of the pattern for Legendrian Whitehead doubling $Wh$; right: diagram of the twisted Whitehead doubling pattern $Wh^{C^{\infty}}_{t=1}$.}
\label{fig:Whitehead_both}
\end{figure}

Now we recall that Legendrian Whitehead doubling operation $Wh$ applied to a Legendrian knot corresponds smoothly to the $t=tb$-twisted smooth Whitehead doubling $Wh^{C^{\infty}}_{t=tb}$of the given Legendrian knot, the diagram of the pattern for $t=1$ appears in Figure \ref{fig:Whitehead_both}. For more details we refer the reader to \cite{EtnyreVertesiLegSat}.
So, our goal becomes to show that $Kh(Wh_t^{C^{\infty}}(\Lambda^i);\mathbb F_2)$, $i=0,\dots,k-1$ are pairwise non-isomorphic. 

We argue the following way.  
Assume that $Kh(Wh_t^{C^{\infty}}(\Lambda^i);\mathbb F_2)$ and $Kh(Wh_t^{C^{\infty}}(\Lambda^j);\mathbb F_2)$ for $i\ne j$ are isomorphic,
then following \cite{Khovanov2000} we see that $Wh_t^{C^{\infty}}(\Lambda^i)$ and $Wh_t^{C^{\infty}}(\Lambda^j)$ have the same Jones polynomials, 
which following \cite{BirmanLinVassiliev} implies that the degree-$3$ Vassiliev invariants $v_3(Wh_t^{C^{\infty}}(\Lambda^i))$ and $v_3(Wh_t^{C^{\infty}}(\Lambda^j))$ are the same. So, instead of comparing $Kh(Wh_t^{C^{\infty}}(\Lambda^i);\mathbb F_2)$ and  $Kh(Wh_t^{C^{\infty}}(\Lambda^j);\mathbb F_2)$  directly, we will show that $v_3(Wh_t^{C^{\infty}}(\Lambda^i))$  and $v_3(Wh_t^{C^{\infty}}(\Lambda^j))$  are pairwise different for $i\neq j$.  Now, we simply need to compute the differences 
$$v_3(Wh_t^{C^{\infty}}(m(9_{46})^i))-v_3(Wh_t^{C^{\infty}}(U))\ \mbox{and}\ v_3(Wh_t^{C^{\infty}}(m(9_{46})^i))-v_3(Wh_t^{C^{\infty}}(m(9_{46})^j))$$
for $1\leq i,j\leq k-1$, $i\neq j$. 
For that we use the computation done by Ichiwara--Wu from \cite[Proposition 5.1]{IchiwaraWu} that
in our notations for a knot $K$ says that
$$v_3(Wh_t^{C^{\infty}}(K))=-a_2(K)+\frac{t^2-t}{2},$$
where $a_2$ is a  $z^2$-coeﬃcient of the Conway
polynomial of $K$.

Using \cite{Knotinfo}, we observe that $a_2(m(9_{46}))=-2$ leading to $a_2(m(9_{46})^i)=-2i$ and $a_2(U)=0$, which implies that for $i\neq j$
\begin{align*}
&v_3(Wh_t^{C^{\infty}}(m(9_{46})^i))-v_3(Wh_t^{C^{\infty}}(U))=2i+\frac{t^2-t}{2}-0-\frac{t^2-t}{2}=2i\neq 0\\
&v_3(Wh_t^{C^{\infty}}(m(9_{46})^i))-v_3(Wh_t^{C^{\infty}}(m(9_{46})^j))=2i+\frac{t^2-t}{2}-2j-\frac{t^2-t}{2}=2(i-j)\neq 0
\end{align*}
Since $v_3(Wh_t^{C^{\infty}}(m(9_{46})^i))-v_3(Wh_t^{C^{\infty}}(U))$ and $v_3(Wh_t^{C^{\infty}}(m(9_{46})^i))-v_3(Wh_t^{C^{\infty}}(m(9_{46})^j))$ do not depend on $t$ and do not vanish for $1\leq i,j\leq k-1$, $i\neq j$, we conclude that $Kh(Wh(\MS(\Lambda^i));\mathbb F_2)$, $i=0,\dots, k-1$, are pairwise non-isomorphic. Then by applying Proposition \ref{non-decompoosableendo} 
we get the desired conclusion for $\Lambda_{fil, k}=Wh(\MS(\Lambda^k)$. This finishes the proof.

\color{black}

\end{document}